 \documentclass[12pt]{amsart}
\usepackage[usenames]{color} 
\usepackage{subfigure}

\usepackage{amssymb, amsmath, amsfonts,amsthm,mathrsfs, mathtools, harpoon}
\usepackage[T1]{fontenc}

\usepackage{graphicx}
\usepackage[colorlinks=true]{hyperref}
  \hypersetup{urlcolor=blue, citecolor=red}

\advance\textheight 0.2cm

\newtheorem{theorem}{Theorem}[section]
\newtheorem{lemma}[theorem]{Lemma}

\newtheorem{proposition}[theorem]{Proposition}

\theoremstyle{definition}

\numberwithin{equation}{section}

\def\LL{{\mathcal L}}

\def\HH{\mathcal H}

\def\RR{{\mathbb R}}

\def\Ln{\mathcal L^n}
\def\m{\mu}

\def\on{\omega_n}
\def\O{\Omega}
\def\rN{\RR^N}

\def\esup { {\rm ess} \sup}

\def\W11{\mathop{W^{1,1}_{\rm loc}}}

\title[P\'olya-Szeg\"o inequality and Comparison results]
{New P\'olya-Szeg\"o-type inequalities \\ and an alternative approach \\ 
to comparison results for PDE's}
\author{F. Brock, F. Chiacchio,  A. Ferone, A. Mercaldo}
\address{Institute of Mathematics --- University of Rostock, 
, Ulmenstr. 69, 18057 Rostock,  Germany,}
\email{friedemann.brock@uni-rostock.de}
\address{Dipartimento di Matematica ---  Universit\`a degli Studi della Campania Luigi Vanvitelli,
Viale Lincoln 5, 81100 Caserta, Italy}
\email{adele.ferone@unina2.it}
\address{Dipartimento di Matematica e Applicazioni ``R. Caccioppoli'' --- Universit\`a degli Studi di Napoli Federico II,
Complesso Monte S. Angelo, via Cintia, 80126 Napoli, Italy}
\email{mercaldo@unina.it}
\email{francesco.chiacchio@unina.it}
\begin{document}

\begin{abstract} We prove some P\'olya-Szeg\"o type inequalities  which involve couples of functions  and their rearrangements. Our inequalities reduce to the classical  P\'olya-Szeg\"o principle when  the two functions coincide. As an application, we give a different proof of  a comparison result for solutions to Dirichlet boundary value problems for Laplacian equations proved in \cite{ADLT}.
%We obtain comparison results for elliptic boundary value problems, related to Steiner symmetrization. Our approach differs from two well-known papers 
%of Baernstein (1992) and Alvino et al (1996), in that it is based on a new generalized Polya-Szeg\"o principle, and it does not make use of level set formulas for second derivatives with to the variables not involved in  the symmetrization. Some applications of the new generalized Polya-Szeg\"o principle with respect to Schwarz symmetrization are also presented. 

\rm 
\\[0.1cm]
\rm 
 
\end{abstract}

\maketitle

%%%%%%%%%%%%%%%%%%%%%%%%%%%%%%%%%%%%%%%%%%
\section{Introduction}
The celebrated P\'olya-Szeg\"o Principle asserts that Dirichlet type integrals do not increase under Schwarz symmetrization. In its simplest form it states that, if  $u$ is a compactly supported function which belongs to $W^{1,2}(\RR^n)$ then also its {\it spherically symmetric rearrangement} $u^\star$ is in $W^{1,2}(\RR^n)$ and
\begin{equation}\label{PS}
\int_{\RR^N} |\nabla u(z)|^2 dz\ge \int_{\RR^N} |\nabla u^\star(z)|^2 dz\,.
\end{equation}
%\marginpar{in questo modo evito di dare la definizione di $\O^\star$}
The interest in the  P\'olya-Szeg\"o principle is due to its multitude of applications 
 in analysis and physics. For instance, it is the main tool in proving isoperimetric inequalities for capacities and of Faber-Krahn type, as well as  apriori estimates for solutions to boundary value problems for PDEs (see e.g.\cite{Ba}, \cite{He}, \cite{Ka},  \cite{Ke}, \cite{Tart} and references therein)
 %the study existence and regularity of of solutions to partial differential problems 
The topic has attracted  the attention of many authors, and it has been developped in various directions since the middle of last century. For instance,  more general functionals of the gradient under different types of symmetrizations have been investigated (see, for example, \cite{BaeTa}, \cite{Br1}, \cite{BMP}, \cite{Bur1}, \cite{CF1}, \cite{E}, \cite{ET}, \cite{Talentiweig} and references therein). More recently, the equality case and the stability in these inequalities have been studied 
%and  also different kind of approach to study these inequalities has been obtained 
(see, for example, \cite{BZ}, \cite{FV}).

In this paper we prove a P\'olya-Szeg\"o type inequality which - unlike the classical case \eqref{PS} - involves {\sl two}  functions  $u,w$ and their rearrangements. 
Our inequality reduces to \eqref{PS} when $u=w$. 
We focus on the Steiner symmetrization, and we will analyze the differences which appear when Steiner symmetrization  is replaced by Schwarz symmetrization. 
\\
The proofs of P\'olya-Szeg\"o type inequalities are tipically based on  the isoperimetric inequality in Euclidean space, while  our approach relies on two further well-known tools from the theory of rearrangements: the   
 {\it Hardy-Littlewood inequality} and the {\it Riesz inequality}. The Hardy-Littlewood 
%very classical result, when Steiner symmetrization is considered,  
states that 
\begin{equation}\label{HL}
\int_{\RR^N} u(z)w(z)dz\le \int_{\RR^N} u^{\#}(z)w^{\#}(z)dz\,,
\end{equation}
for any couple of measurable nonnegative functions. Here $u^\#$ and $w^\#$ are the Steiner rearrangements of $u$ and $w$ respectively defined in Section \ref{sec:defsch}. 
 
The main result of the paper is
\begin{theorem} \label{main}
Assume that $u$ and $w$ are Lipschitz-continuous nonnegative functions with compact support defined in $\RR^N$ and 
\begin{equation}\label{HL=}
\int\limits_{\rN} u(z)w(z)dz=\int\limits_{\rN} u^\#(z)w^\#(z)dz\,.
\end{equation}
Then the following inequalities hold
\begin{equation}\label{PSgx}
\int_{\O}\nabla_xu(z)\cdot\nabla_xw(z)\,dz\ge \int_{\O^\#}\nabla_xu^\#(z)\cdot\nabla_xw^\#(z)\,dz\,,
\end{equation}
and that, for each $i=1, \dots, m$
\begin{equation}\label{PSgy}
\int_\O u_{y_i}(z)\cdot w_{y_i}(z)\, dz \ge \int_{\O^\#}u^\#_{y_i}(z)\cdot w^\#_{y_i}(z)dz\,
\end{equation}
and, hence
\begin{equation}
\label{PSg}
\int\limits_{\rN} \nabla u(z)\cdot \nabla w(z)dz
 \ge  \int\limits_{\rN} \nabla u^\#(z)\cdot \nabla w^\#(z)dz\,.
\end{equation}
\end{theorem} 

\noindent Note that, if $u=w$, then equation \eqref{HL=} is in force, since symmetrization preserves the $L^2$ norm. As previously mentioned,  in such case \eqref{PSg} reduces to the standard P\'olya-Szeg\"o inequality \eqref{PS}. %Moreover, if $w= \Phi (u)$, with   

\noindent The proof of Theorem 1.1 is based on a  discretization of the gradient and the  Riesz Inequality.
%We emphasize that  such result hold true under strong assumption on the regularity of  {\it both} functions $u,w$ satisfying \eqref{HLS=} exists. This assumption allows us to linearize their gradients, and then to obtain the thesis by  {\it Riesz Inequality}. In order to avoid regularity assumption on both functions, we may weaken such assumption for one of them, provided to strengthen it for the other to obtain the following result.
\\
Inequality \eqref{eq:wPS} in our next Theorem is related to \eqref{PSg}.  %related result is
% next result 
The difference is 
%We emphasize that  
%such result can be applied for instance when  %$w$ is any increasing function of $u$. 
%\noindent We remark that 
%we are not be able to  remove the regularity assumptions on both $u$ and $w$. Indeed in the next result 
that we allow $w$ to be {\sl not } weakly differentiable, but instead we require more regularity for $u$.
\begin{theorem} \label{wPS}
%\marginpar{Le ipotesi su $u$ devono assicurare la convergenza uniforme del rap incrementale alla derivata fino all'ordine 2. Questo si ha laddove c'\'e uniforme continuit\'a, cio\'e in compatti.}
Let $\O$  be  a  bounded domain of $\RR^N$ and let $u\in C^2(\O)\bigcap C(\overline{\O})$ be a nonnegative function satisfying   $u=0$ on $\partial \O$. 
Further, let $W\in W_0^{1,\infty}(\O ^\#)$ be  a nonnegative function such that $W=W^\#$.  Then, if $w$ is any function satisfying \eqref{HL=} with $w^\#=W$, we have that
\begin{equation}\label{eq:wPS}
-\int_{\O } w(x)\Delta u(x)dx\ge \int_{\O^\#}\nabla u^\#(x) \cdot\nabla W(x)dx\,.
\end{equation}
\end{theorem}

As an application of the previous two Theorems, we recover a comparison result proved in \cite{ADLT}, (see also \cite{Bae},  \cite{BK}, \cite{ChRic}, \cite{FM}, \cite{FM2} and the references therein). More precisely, we consider the following linear homogeneous Dirichlet problem 
\begin{equation}\label{prSt}
\begin{cases}
-\Delta\, u=f &\mbox{in } \O\\
u=0 &\mbox{on } \partial \O\,,
\end{cases}
\end{equation}
where $\O $  is a  bounded domain  of $\RR^{N}$. We decompose every $z\in\RR^N$ by $z=(x,y)$ with $x\in\RR^n\,,y\in\RR^m$ and $n+m=N$. Accordingly, let $\Omega _y = \{ x\in \mathbb{R}^n :\, (x,y)\in \Omega \} $ denote the $y$-section of $\Omega $.   By $B_r(0)$ we denote  the $n$-dimensional ball centered at the origin with radius  $r$, and by $\O^\#$  the subset of $\RR^N$ such that, for any $y\in\RR^m$, its $y$-section $(\O^\#)_y=\{x\in\RR^n:\, (x,y)\in \O^\#\}$ is the $n$-dimensional ball centered at zero which has the same ${\mathcal L}^n$-measure as  $\O _y$.  
Then the following result holds.

\begin{theorem}\label{compnonlinSt}
Let $\O $ be a  bounded domain  of $\RR^{N}$ satisfying the exterior sphere condition and $f\in L^q(\Omega)$, $q>\frac N 2$. Further, let $u\in W^{1,2}_0(\O)$ be a weak solution to problem \eqref{prSt}, and let $v\in W^{1,2}_0(\O^\#)$ be the weak solution to the symmetrized problem
\begin{equation}\label{PstarSt}
\begin{cases}
 -\Delta v=f^\# &\mbox{ in } \O^\#\,,\\
 v=0 & \mbox{ on }\partial\O^\#\, .
\end{cases}
\end{equation}
Then we have for all $r\in [0,\infty)$ and for a.e. $y\in\RR^m$
$$
\int_{B_r(0)} u^\# (x,y)dx\le\int_{B_r(0)}  v^\#(x,y)dx \,.
$$
\end{theorem}
%\marginpar{ho scritto cos la tesi, altrimenti dovevamo dare la definizione del riordinamento decrescente. mi sembra che sia del tutto equivalente ma controllate!}
\noindent Note that all the previous results hold also for Schwarz symmetrization, with appropriate modifications. In such a case  the absence of the $y$-variables allows to recover  the well-known {\sl pointwise} comparison result which is due to Talenti (see \cite{Talentinonlin}). 
% comparison between the rearrangements of the solutions is pointwise. 
The case of Schwarz symmetrizaton will be treated in Section \ref{sec:der} where also nonlinear problems are considered.
% All these differences will be treated in Section \ref{sec:der} and Section \ref{sec:comparison}.

%%%%%%%%%%%%%%%%%%%%%%%%%%%%%%%%%%%%%%%%
\section{Notation and preliminary results}
\label{sec:defsch}
In this section we introduce some notations, and we recall some well-known results which will be used in the sequel.

Let  $\RR^{N}$, $N \ge 1$, be the Euclidean space and let $E$ be a measurable subset of $\RR^N$.   The $N$-dimensional Lebesgue measure of the set $E$ is denoted by $\LL^{N}(E)$, while for any $d\ge 0$, $\HH^d(E)$ denotes its $d$-dimensional Hausdorff measure. The notation $|\cdot|$ denotes the standard Euclidean norm, independently from the dimension of the space.    
 
Let $\O$ be an open  subset of $\RR^{N}$, $N\ge 1$, and let $u$ be a nonnegative measurable  function on $\O$. Its {\sl distribution function}  is given by
$$
\m_u(t)=\LL^{N}\left (\{x\in\O :u(x)>t\}\right )\qquad t\in[0,+\infty )\,,
$$
and its   {\sl decreasing rearrangement} is defined as
$$
u^*(s)=\sup\{t\ge 0: \m_u (t)>s\}\, , \quad s\in (0,\LL^{N}(\O)]\,.
$$
We denote by $\O^\star$ the  ball of $\RR^n$ centered at the origin and having the same ${\mathcal L}^n$-measure as $\O$.  The  {\sl Schwarz rearrangement} of $u$, is given by
$$
u^\star(x)=u^*(\on|x|^N)\qquad x\in\O^\star\,,
$$
where  $\on$ is the measure of the $n$-dimensional unit ball. 

If $N\ge 2$, let $n,m\in \mathbb{N} $ be such that $n+m=N$, and decompose every $z\in\RR^{N}$  by  $z=(x,y)$, with $x\in\RR^n$ and $y\in\RR^m$. Accordingly, the gradient $\nabla u$ of  a function $u$ is the pair $ (\nabla_xu, \nabla_yu)$,  where
$\nabla_x u=\left (\frac{\partial u}{\partial x_1}, \dots, \frac{\partial u}{\partial x_n}\right )$ and
$\nabla_y u = \left (\frac{\partial u}{\partial y_1}, \dots , \frac{\partial u}{\partial y_m}\right )$.
\\
% Let $\O$ be an open  subset of $\RR^{N}$. 
For any $y\in\RR^m$, let $\O_y$ be the $y$-section of $\O$ which is defined by 
$$
\O_y :=\{x\in\RR^n: (x,y)\in\O\}\, , \quad y\in \RR^m\,.
$$
%where
%$$
%\Omega ':= \{ y\in \mathbb{R}^m :\, (x,y) \in \Omega 
%\mbox{ for some $x\in \mathbb{R} ^n$} \} .
%$$   
%Let $u$ be a nonnegative measurable  function on $\O$. 
The {\sl distribution function (in codimension $n$)} of $u$ and its   {\sl decreasing rearrangement (in codimension $n$)} are defined as
$$
\m_u(t,y)=\Ln\left (\{x\in\O_y :u(x,y)>t\}\right ), \qquad (t,y)\in[0,+\infty )\times\RR^m\,,
$$
and
$$
u^*(s,y)=\sup\{t\ge 0: \m_u (t,y)>s\}\, , \quad (s,y)\in (0,\Ln(\O_y)]\times\RR^m\, ,
$$
respectively.
By $\O^\#$  we denote the open set in $\RR^{N}$ such that, for any $y\in\RR^m$, its $y$-section $(\O^\#)_y$ is the $n$-dimensional ball   centered at the origin
% lying in the hyperplane $\{(x,y): x\in\RR^m\}$ 
and having the same ${\mathcal L}^n$-measure as $\O_y$. The  {\sl Steiner symmetrization (in codimension $n$)} of $u$, is given by
\begin{equation}\label{stsc}
u^\#(x,y)=(u(\cdot,y))^\star(x)=u^*(\on|x|^n,y) \qquad (x,y)\in\O^\#\,.
\end{equation}

It is well known that if $u\in W_0^{1,p}(\O)$, for some $1\le p\le \infty$, then also $u^\#\in W_0^{1,p}(\O^\#)$, and the $L^p$ norm is preserved while the $W^{1,p}$ norm is reduced (see for example \cite{BaeTa, BLM, Br1, Bur2} and the references therein).

The Hardy-Littlewood inequality with respect to the Schwarz rearrangement states that if $u$ and $w$ are nonnegative measurable function on a bounded  open set $\O$ of $\RR^n$, $n\ge 1$, then 
\begin{equation}\label{HLS}
\int_\O u(z)w(z)dz\le \int_{\O^\star}u^\star(z)w^\star(z)dz\,.
\end{equation}
Furthermore, the  Riesz inequality states that
\begin{equation}\label{RS}
\int_{\mathbb{R}^{2N} }\!\!u(x)w(z)h(x-z)\,dxdz
 \le \int_{\mathbb{R}^{2N} }\!\! u^{\star}(x)w^{\star}(z)h^\star (x-z)\, dxdz\,.
\end{equation}
%Note that (\ref{RS}) reduces to \eqref{HLS}  when $h\equiv1$. 
for any triple $u,w,h$ of  nonnegative measurable functions on $\mathbb{R}^N $ for which the right-hand side is finite.  

\noindent In the following we are interested in the situation where equality in \eqref{HLS} is achieved. Let $u$ and $W=W^\star$ be two given nonnegative measurable functions, defined in $\Omega$ and $\Omega^\star$, respectively. We will say that a function $w$, satisfying $w^{\star} =W$, is an {\it extremal} for \eqref{HLS}, if it produces equality in \eqref{HLS}, that is
\begin{equation}\label{HLS=}
\int_\O u(x)w(x)dx= \int_{\O^\star}u^\star(x)W(x)dx\,.
\end{equation}
Extremals of \eqref{HLS} have been completely characterized (see, for example \cite{ALT1, Bur2, C1, CiFe1}). In particular,  an extremal $w$ always exists. However, it is not unique in general.   Furthermore,  equality (\ref{HLS=}) holds if and only if 
 the level sets  of $u$ and the level sets of $w$ are  mutually nested, that is,  for any choice of values $t, \tau$ there holds 
\begin{eqnarray*}
%\label{CC1}
\mbox{either } & & \{x: u(x)>t\} \subset \{x: w(x)>\tau\},
\\
\mbox{ or } & &  
\{x: w(x)>\tau \} \subset \{x: u(x)>t\}. 
\end{eqnarray*}
An equivalent condition is  
\begin{equation*}
%\label{C2}
\left( u(x ) - u(x') \right) \left( w(x)-w(x') \right) \geq 0 \ \mbox{ for a.e. $(x,x') \in \Omega \times \Omega $\,.}
\end{equation*} 
The Hardy-Littlewood inequality \eqref{HL} for Steiner symmetrization can be easily recovered by the one for Schwarz symmetrization. Indeed, if $N\ge 2$,  we easily deduce from \eqref{HLS},
$$
\int_{\O_y} u(x,y)w(x,y)dx\le \int_{\O_y^\star}(u(\cdot,y))^\star(x)(w(\cdot,y))^\star(x)dx\quad \mbox{for a.e. } y\in\RR^m\,, 
$$
which immediately implies \eqref{HL}, thanks to \eqref{stsc}. 

Finally, we recall the following well-known result of \cite{ALT1}.
\begin{proposition}\label{ALT} Let $\O$ be a bounded domain of $\RR^{N}$, and let $u,v\in L^1 (\O) $  be two nonnegative functions. Then we have  for a.e. $y\in \RR^m$, 
%\marginpar{visto quello che ho scritto nella introduzione vogliamo scrivere il riordinamento decrescente o lasciare quello di steiner come nell'introduzione?}
$$ \int_0^s u^*(s,y)ds\le \int_0^s v^*(s,y)ds \quad \mbox{for } s\in[0, \Ln(\O_y)]\, ,
$$
if and only if
$$\int_{\O^\#}u^\#(x,y)h(x,y)dxdy\le\int_{\O^\#}v^\#(x,y)h(x,y)dxdy\,,$$
for every nonnegative function $h=h^\#$ belonging to $L^\infty(\O^\#)$.

\end{proposition}

\section{New P\'olya-Szeg\"o type inequalities for Steiner symmetrization}
In this section we prove the Theorems 1.1 and 1.2.  
% Let us begin by proving Theorem \ref{main}.
\medskip

%Theorem \ref{main} and Theorem \ref{wPS} will follow immediately from the following two results.
%\begin{theorem}\label{mainweak}
%Under the same assumption of Theorem \ref{main}, we have
%\begin{equation}\label{PSgx}
%\int_{\O}\nabla_xu(z)\cdot\nabla_xw(z)\,dz\ge \int_{\O^\#}\nabla_xu^\#(z)\cdot\nabla_xw^\#(z)\,dz\,,
%\end{equation}
%and that, for each $i=1, \dots, m$
%\begin{equation}\label{PSgy}
%\int_\O u_{y_i}(z)\cdot w_{y_i}(z)\, dz \ge \int_{\O^\#}u^\#_{y_i}(z)\cdot w^\#_{y_i}(z)dz\,.
%\end{equation}
%\end{theorem} 
\begin{proof} [{\it Proof of Theorem  \ref{main}}]  Let us first show inequality \eqref{PSgx}.
For convenience, we extend  $u$ and $w$ by zero outside of $\O$. Let $h\in\RR^n$ be such that $|h|\le 1$. Since $u$ is a Lipschitz function, it is differentiable a.e. Hence
\begin{equation}\label{point}
\lim_{\varepsilon \to 0}
\frac{u(x+\epsilon h,y )- u(x,y)}{\varepsilon}=\nabla_x u(x,y)\cdot h\,, \quad \hbox{for a.e. }x\in\RR^n\, ,
\end{equation}
and 
\begin{equation}\label{point2}
\frac{|u(x+\epsilon h,y )-u(x,y)|}{\varepsilon}\le \|\nabla_x u\|_{L^\infty(\RR^N)}\,|h|\,,\qquad 0<\epsilon<\epsilon_0\,,
\end{equation}
for a suitable  $\epsilon_0>0$, and analogously for $w$. 
%\marginpar{ci vuole questo $\epsilon_0$ quando estendiamo a zero le funzioni?}
\noindent Let $B_1(0) $ denote the unit ball in $\mathbb{R}^n $, and let $\phi\in C^\infty_0(B_1(0) )$ be a radial and radially nonincreasing function.
%
%\noindent Next, let $\phi\in C^\infty_0(\RR^n)$ be a radial function supported in the unit ball $B_1(0)$ of $\RR^n$.
%, centered at the origin,
By  \eqref{point},  \eqref{point2} and the Dominated Convergence Theorem it follows that
\begin{align} 
\label{grad}
  \lim_{\varepsilon \to 0}\int_{\RR^N}\int_{B_1(0)}\!\!
&\frac{(u(x+\epsilon h,y )-u(x,y))\,(w(x+\epsilon h,y )-w(x,y))}{\varepsilon^2}\,\phi (h)\, dhdxdy\notag
\\
 &= \int_{\RR^N}\int_{B_1(0)}
(\nabla_x u(z)\cdot h)(\nabla_x w(z)\cdot h)\, \phi (h) \,dhdz \\
\nonumber
 & =  \displaystyle\sum_{i,j=1}^n
\int_{\RR^N}u_{x_i}(z) \,w_{x_j}(z)\left (\int_{B_1(0)}
\phi(h)\,h_ih_j \,dh\right )dz\,. 
\end{align}
Since $\phi$ is radial, we deduce that
$$
\int_{B_1(0)}\phi (h)\,h_ih_j \, dh=0 \quad \mathop{\rm for\>} i\neq j\,,
$$
and for $i=1,\dots , n\, $ we have
\begin{equation}\label{C}
\int_{B_1(0)}\phi (h)\,h_i^2\,dh=\frac{1}{n}\int_{B_1(0)}\phi (h)\, |h|^2 \,dh=\frac {C}n\,,
\end{equation} 
where 
\begin{equation}\label{C1}
 C:= \int_{B_1(0)}\phi (h)\, |h|^2 \,dh\,.
\end{equation}
From \eqref{grad}-\eqref{C1} we obtain
\begin{align}
\label{PSx}
&\int_{\RR^N}\nabla_xu(z)\cdot\nabla_xw(z)\,dz\\
\nonumber 
= \frac Cn\lim_{\epsilon\rightarrow 0}
\int_{\RR^N}\!\! \int_{B_1(0)}&\!\!\!\!
\frac{(u(x+\epsilon h, y )-u(x,y))(w(x+\epsilon h, y )-w(x,y))}{\epsilon^2}\phi(h)dhdxdy\,.
\end{align}
%\marginpar{io credo che sia meglio considerare il prolungamento a zero di $u$ su tutto lo spazio perch\'e non possiamo integrare su $\O$ la funzione $u(x+\epsilon h)$}
%(\ref{C1}).
%and
%\begin{align}\label{PSy}
%\int_{\rN}&\nabla_yu\cdot\nabla_yv\,dxdy\\
%&\notag\\
%&=\frac CN\lim_{\epsilon\rightarrow 0}\int_{\rN}\ \int_{B_1}\!\!\!\!
%\frac{(u(x+\epsilon h, y )-u(x))(v(x+\epsilon h, y)-v(x, y))}{\epsilon^2}\phi(h)dhdxdy
%\end{align}
On the other hand, since $\phi = \phi ^{\star}$ we get  by Riesz' inequality for a.e. $y\in\RR^m$ 
\begin{align*}
\int_{\RR^m} \int_{B_1(0)} &u(x+\epsilon h, y)w(x,y)\phi(h)dhdx\\
&\le\int_{\RR^m} \int_{B_1(0)} (u(\cdot, y))^\star (x+\epsilon h)(w(\cdot,y))^\star\phi(h)dhdx\, .
\end{align*}
By integrating this w.r.t. $y$ and recalling the definition of \eqref{stsc}, this leads to
\begin{align}\label{riesz1}
\int_{\RR^N} &\int_{B_1(0)} u(x+\epsilon h,y)w(x, y)\phi(h)dxdydh\\
&\le\int_{\RR^N} \int_{B_1(0)}u^\#(x+\epsilon h,y)w^\#(x, y)\phi(h)dxdydh\,.\notag
\end{align}
Similarly, we deduce
\begin{align}\label{riesz1}
\int_{\RR^N} &\int_{B_1(0)} u(x,y)w(x+\epsilon h, y)\phi(h)dxdydh\\
&\le\int_{\RR^N} \int_{B_1(0)}u^\#(x,y)w^\#(x+\epsilon h, y)\phi(h)dxdydh\,.\notag
\end{align}
Furthermore, since $u$ and $w$ satisfy  \eqref{HL=}, we find
\begin{align}\label{eq1}
\int_{\RR^N} u(x+\epsilon h, y)&w(x+\epsilon h, y)\, dxdy
\\
\notag &=
\int_{\RR^N}  u^\# (x+\epsilon h, y)w^\#(x+\epsilon h, y)\, dxdy\,,
\end{align}
and similarly,
\begin{equation}
\label{eq2}
\int_{\RR^N}  u(x, y)w(x, y) \, dxdy=
\int_{\RR^N} u^\# (x, y)w^\#(x, y) \,dxdy.
\end{equation}
Collecting \eqref{riesz1}-\eqref{eq2},  we get
\begin{eqnarray*}
 & & \int_{\rN}\ \int_{B_1}\!\!\!\!
\frac{(u(x+\epsilon h, y)-u(x, y))(w(x+\epsilon h, y)-w(x, y))}{\epsilon^2}\phi(h)dhdxdy\\
&
\ge& \int_{\rN}\ \int_{B_1}\!\!\!\!
\frac{(u^\#(x+\epsilon h, y)-u^\#(x, y))(w^\#(x+\epsilon h, y)-w^\#(x, y))}{\epsilon^2}\phi(h)dhdxdy\,.
\end{eqnarray*}
Finally, passing to the limit $\epsilon \to 0$ and using \eqref{PSx} we obtain \eqref{PSgx}.

It remains to prove \eqref{PSgy}. Fix $i\in \{ 1, \ldots ,m\} $, and let $e_i$ denote the unit vector of $\mathbb{R} ^m$ in the positive $y_i $-direction. Then
\begin{eqnarray}
\label{yi} 
& & \int_{\rN}u_{y_i }(x,y) w_{y_i }(x,y)\,dxdy
\\
\nonumber 
&= & \lim_{\epsilon \to 0} 
\int_{\rN} \frac{
(u(x, y+ \epsilon e_i )-u(x,y))(w(x, y+ \epsilon e_i )-w(x,y))}{\epsilon ^2} \,dxdy\, .
\end{eqnarray}
An analogous relation holds  for $u^\# $ and $w^\# $ in place of $u$ and $w$. Now,
(\ref{HL}) yields
\begin{equation}
\label{yi1}
\int\limits_{\rN} u(x, y+ \epsilon e_i )w(x, y ) \,dxdy
 \leq  
\int\limits_{\rN} u^\#(x, y+ \epsilon e_i )w^\#(x, y ) \,dxdy 
\end{equation}
and 
\begin{equation}
\label{yi2}
\int\limits_{\rN} u(x, y) w(x,y+ \epsilon e_i ) \,dxdy
 \leq 
\int\limits_{\rN} u^\#(x, y)w^\# (x,y+ \epsilon e_i ) \,dxdy ,
\end{equation}
while (\ref{HL=}) gives
\begin{eqnarray}
\label{yi3}
 & & \int\limits_{\rN} u(x, y+ \epsilon\, e_i ) w(x, y+\epsilon\, e_i  ) \,dxdy
 \\ 
\nonumber
 & = &  
\int\limits_{\rN} u^\#(x, y+ \epsilon e_i )w^\#(x, y +\epsilon e_i ) \,dxdy\, . 
\end{eqnarray}
Now inequality (\ref{PSgy}) follows from 
(\ref{yi})-(\ref{yi3}).
\end{proof}

For the proof of  Theorem \ref{wPS} we will need the following result.
\begin{theorem} 
%\marginpar{Le ipotesi su $u$ devono assicurare la convergenza uniforme del rap incrementale alla derivata fino all'ordine 2. Questo si ha laddove c'\'e uniforme continuit\'a, cio\'e in compatti.}
Under the assumptions of Theorem \ref{wPS}, there holds
\begin{equation}\label{eq:wPSx}
-\!\int_{\O } w(x,y)\Delta_x u(x,y)dxdy\ge \int_{\O^\#} \nabla_x W(x,y)\cdot\nabla_x u^\#(x,y)dxdy\, ,
\end{equation}
and for any $i=1, \dots, m$,
\begin{equation}\label{eq:wPSy}
-\int_{\O } w(x,y) u_{y_iy_i}(x,y)dxdy\ge \int_{\O^\#}  W_{y_i}(x,y)\cdot u_{y_i}^\#(x) dxdy\,.
\end{equation}
\end{theorem}
\begin{proof}   
Let us first prove inequality \eqref{eq:wPSx}.  Let  $\phi\in C_c^\infty(\RR^n)$ be a radial function, compactly supported in the unit ball of $\RR^n$ and let $C$ be the constant defined in \eqref{C}. Then
\begin{align}\label{lap1}
\frac Cn&\Delta_x u(x,y) \\
&= \lim_{\varepsilon \to 0} \int_{\RR^n}\!\frac{u(x+\varepsilon h,y)-2u(x,y)+u(x-\varepsilon h,y)}{\varepsilon^2} \, \phi(h)\, dh\,,\notag
\end{align}
for any $(x,y)\in\O$, with uniform convergence on compact subsets of $\O$. To see that, let 
 $(x,y)\in\Omega$ and choose $\varepsilon_0>0$ small enough such that $x+\varepsilon h\in\O $ and $x-\varepsilon h\in \Omega $ for every $\varepsilon\in (0, \varepsilon_0 )$ and for every $h\in B_1(0)$. Then a Taylor expansion gives
\begin{align*}
D^2_x u(x,y,\epsilon)&=\frac {u(x+\varepsilon h,y)-2u(x,y)+u(x-\varepsilon h,y)}{\varepsilon^2}\\
&=\sum_{i,j=1}^nu_{x_ix_j}(x,y)h_ih_j+o(1)\,,
\end{align*}
with uniform convergence on compact subsets of $\O$.  (Here: $\lim_{\varepsilon \to 0}o(1)=0\,.$)
Since  \eqref{C} holds for every
 nonnegative radial function  $\phi$, we have
\begin{align*}
 \int_{B_1(0)}
D^2_x u(x,y,\epsilon)\phi (h)dh 
 &=\sum_{i=1}^n
\int_{B_1(0)}u_{x_ix_i}(z)h_i^2\phi(h)dh+o(1) \\&= \frac C n \Delta u (x,y) + o(1)\,,
\end{align*}
from which \eqref{lap1} immediately follows on letting $\epsilon$ go to zero.
For any $\delta>0$, let $\delta'$ be such that
$$
\{W>\delta\}\subset \subset  \{u^\#>\delta'\}\,,
$$
and set $W_\delta=(W-\delta)_+$ and  $w_\delta=(w-\delta)_+$.  It is easy to check that  $w_\delta$ is compactly supported,  $(w_\delta)^\#=W_\delta$ and 
$$
\int_\O u(z)w_\delta(z)dz=\int_{\O^\star} u^\#(z)W_\delta (z)dz\,.
$$
In view of the uniform convergence on compact sets in \eqref{lap1}, we deduce
\begin{align}\label{lap2}
 -\lim_{\varepsilon \to 0}\int_{\Omega}\int_{B_1(0)}D^2_x u(x,y,\epsilon)&w_\delta(x,y)\phi(h) dxdydh\\
&=-\int_{\Omega}w_\delta(z)\Delta_x u(z) \,dz\notag\,.
\end{align}
On the other hand, arguing as in the proof of Theorem \ref{main}, we get
\begin{align}\label{lap3}
\int_{\Omega}\int_{B_1(0)} &D^2_xu(x,y,\epsilon)  w_\delta(x,y)\phi(h) dxdy dh\\
\le &\int_{\Omega^\#}\int_{B_1(0)}D^2_xu^\#(x,y,\epsilon) \, W_\delta(x,y)\phi(h)\, dxdy dh\notag\,.
\end{align}
Furthermore, a standard change of variables gives
\begin{align}\label{lap34}
\int_{\O^\#}\int_{B_1(0)}D^2_xu^\#(x,y,\epsilon) W_\delta(x,y)\phi(h)&\, dxdy dh\\
\notag =\!\int_{\Omega^\#}\!\!\int_{B_1(0)}\!\!\frac{(u^\#(x+\varepsilon h,y)-u^\#(x,y))(W_\delta(x+\varepsilon h,y)-W_\delta(x,y))}{\epsilon^2}& \phi(h)\, dxdy dh\,.
\end{align}
Finally, collecting \eqref{lap2}-\eqref{lap34} we obtain
%\begin{align}\label{lap4}
% -\lim_{\varepsilon \to 0} \int_{\Omega_\delta}\int_{B_1(0)}&\frac{u(x+\varepsilon h,y)-2u(x,y)+u(x-\varepsilon h,y)}{\varepsilon^2}  w_\delta(x,y)\phi(h)\, dxdy dh\notag\\
% &\ge\frac Cn\int_{\Omega^\#}\nabla u^\#(x,y)\cdot \nabla W_\delta(x,y)\, dxdy\,.
% \end{align}
%Combining \eqref{lap2} and \eqref{lap4}, we get
$$
-\int_{\Omega}w_\delta(z)\Delta_x u(z) \,dz\ge \int_{\Omega^\#}\nabla u^\#\cdot \nabla W_\delta\, dz\,,
$$
which leads to the thesis on letting $\delta$  go to zero. Inequality \eqref{eq:wPSy} follows in a similar way.
%\marginpar{non ho messo la dim perche (se e vera!!!) allora e identica al caso delle variabili rispetto alla x.}
\end{proof}
Finally, it remains to prove the comparison result of Theorem \ref{compnonlinSt}. 

\begin{proof} [{\it Proof of Theorem  \ref{compnonlinSt}}]  By a standard approximation argument it is enough to prove our result when the datum $f$ is analytic (see, for example, \cite{CM}, \cite{Ch}). This implies that the solution $u$ is analytic too. 
%Suppose that $f$ is sufficiently smooth, so that $u$ is a classical solution to the problem \eqref{prSt}. 
Let $h\in C^\infty (\O^\#)$, be such that $h=h^\#$ and consider the solution to the problem
\begin{equation}\label{st0}
\begin{cases}
-\Delta W = h&\mbox{ in } \O^\#\\
W=0 &\mbox{on } \partial\O^\#\,.
\end{cases}
\end{equation}
We have that $W\in C^\infty(\O^\#)$,  $W=W^\#$ and, by \eqref{PstarSt},  we get
\begin{equation}\label{st1}
\int_{\O^\#} f^\#(x,y)\, W(x,y)dxdy= \int_{\O^\#}\nabla v(x,y)\cdot\nabla W(x,y)dxdy\,.
\end{equation}
On the other hand, if $w$ is a function satisfying \eqref{HL=} such that $w^\#=W$, then by Theorem \ref{wPS} and \eqref{prSt} we get
\begin{align}\label{st2}
\int_\O f(x,y)w(x,y)dxdy&=-\int_\O \Delta u (x,y)w(x,y)dxdy\notag\\
&\ge \int_{\O^\#} \nabla u^\#(x,y)\cdot \nabla W(x,y)dxdy\,.
\end{align}
Collecting \eqref{st1} and \eqref{st2}, we obtain by the Hardy-Littlewood inequality,
\begin{equation}\label{st3}
\int_{\O^\#}\left [ \nabla u^\#(x,y)-\nabla v(x,y)\right ]\cdot\nabla W(x,y)dxdy\le 0\,,
\end{equation}
or, equivalently, by \eqref{st0},
\begin{align*}
\int_{\O^\#} \left [u^\#(x,y)-v(x,y)\right ]&\left (-\Delta W(x,y)\right )dxdy\\
&=\int_{\O^\#} \left [u^\#(x,y)-v(x,y)\right ]h(x,y)dxdy\le 0\,.
\end{align*}
By the arbitrariness of $h$ we deduce the thesis, applying Proposition \ref{ALT}.
\end{proof}

%%%%%%%%%%%%%%%%%%%%%%%%%%%%%%%%%%%%%%%%%
\section{Schwarz symmetrization for nonlinear problems: a new approach}
\label{sec:der}

In this Section we will adapt and modify the previous tools for Schwarz symmetrization. In this case, the gradients of the functions $ u^\star$ and $ w^\star$ are parallel, a fact which simplifies the approach a great deal. An analogue of Theorem \ref{main} for Schwarz symmetrization states as follows.
\begin{theorem} \label{mainSch}
Let $\O$ be an open set of $\RR^n$, $n\ge 1$, and let  $u,w\in W_0^{1,\infty}(\O)$  be  nonnegative functions such that
\begin{equation}\label{HLS2=}
\int_\O u(x)w(x)dx= \int_{\O^\star}u^\star(x)w^\star(x)\, dx\,,
\end{equation}
 then
\begin{equation}\label{PSg1}
\int_{\RR^n} \nabla u(x)\cdot \nabla w(x)\, dx\,\ge\, \int_{\RR^n} |\nabla u^\star(x)|\cdot | \nabla w^\star (x)|dx\,.
\end{equation}
\end{theorem} 
As we already mentioned in Section \ref{sec:defsch}, functions which satisfy \eqref{HLS2=}, have been completely characterized. In particular,  
it has been observed in \cite{CiFe1}, Theorem 1.1, that the extremal functions $w$ in \eqref{HLS2=} are unique if and only if $u^*$ is strictly monotone.
%the unique responsible for non-uniqueness of extremal functions $w$ in \eqref{HLS2=} (see Theorem 1.1 in \cite{CiFe1}).
As a consequence, any extremal $w$ is uniquely determined outside the flat zones of $u$, and it is given by 
\begin{equation}\label{rapw}
w(x)=w^*(\m_u (u(x)))\,,
\end{equation}
for a.e. $x$, such that $u(x)$ is a point of continuity for $\mu_u$. 
In particular, we deduce the uniqueness of such extremal if $w^*$ is constant where $u^*$ is constant.

Furthermore, 
%Concerning the regularity of such an extremal, 
if $w^*$ is a smooth function, then the classical result of Vall\' ee-Poussin on differentiability of composite functions tells us, that any function satisfying \eqref{HLS2=}   is differentiable at any $x$ such that $u(x)$ is a point of differentiability for $\mu_u$ and
 \begin{equation}\label{DW}
 \nabla w(x)=(w^*)'(\mu_u(u(x)))\, (\mu_u)'(u(x))\,\nabla u(x)\,.
 \end{equation}
The differentiability properties of $\mu_u$ have been studied in \cite{BZ, CF1}. In particular, if $u\in W^{1,1}_0(\O)$ and 
\begin{equation}\label{diffustar}
\Big |\Big \{s\in (0,\Ln(\O)): \, (u^*)'(s)=0
\quad 0<u^*(s)<\esup u\Big \}\Big |=0\,,
\end{equation}
then $\mu_u\in W^{1,1}((0,+\infty))$ and
\begin{equation}\label{dmut}
\mu_u'(t)=\frac 1 { ( u^*)'(s)_{|s=\mu_u(t)}}=-\frac {n\on^{\frac 1 n}\mu_u(t)^{1-\frac  1 n} } {|\nabla u^\star (x)|_{\{x:u^\star (x)=t\}} }\,.
\end{equation}

A stronger assumption than \eqref{diffustar}, which ensure the differentiability of $\mu_u$, is 
\begin{equation}\label{diffu}
\Big |\Big \{x\in \O: \, |\nabla u(x)|=0
\quad 0<u(x)<\esup u\Big \}\Big |=0\,.
\end{equation}
The following Lemma gives  sufficient conditions on $w^*$ which ensure the uniqueness and regularity of the  extremal $w$.
 
 \begin{lemma}\label{regularity}
Let $\O$ be an  bounded domain of $\RR^n$ and let $u\in W^{1,p}_0(\O)$, $1\le p<\infty$. Let $W: (0, \Ln(\O)] \to \mathbb{R} $ be a nonincreasing function belonging to $ W^{1,p}(a,\Ln(\O))$ for every $a>0$, such that $W(\Ln(\O))=0$
and 
\begin{equation}\label{key}
-W'(s)\le C(-u^*)'(s)\qquad for\>a.e.\> s\in (0,\Ln(\O))\,,
\end{equation}
for some positive constant $C$.
% (DEPENDS ON SOMETHING ?). 
Then there exists only one  function $w\in W^{1,p}_0(\O)$ satisfying $w^{\star} =W$ and \eqref{HLS=}.
Moreover, $W \circ \mu _u : [0, +\infty ) \to \mathbb{R} $ is Lipschitz-continuous   and 
\begin{equation}\label{Dw}
\nabla w(x)=W'(\mu_u(u(x)) ) (\mu_u) '(u(x))\nabla u(x)\quad for\>\>a.e.\>\>x\in\O\,.				
\end{equation}
\end{lemma}
\begin{proof}
Hypothesis \eqref{key} ensures that $W$ is constant where $u^*$ is constant, so that the uniqueness of the extremal satisfying \eqref{HLS=} easily follows. Moreover such an extremal is given by \eqref{rapw}, so it remains to prove that $w\in W^{1,p}_0(\O)$. 

To this aim, let us consider $\phi\in C^1_c(0,+\infty)$.  By the absolute continuity of $W$ and since $W(\Ln(\O))=0$, we get
$$
\int_0^{+\infty} W(\mu_u(t))\phi'(t)dt=\int_0^{+\infty} \left (\int_{\mu_u(t)}^{\Ln(\O)} (-W'(s))ds \right ) \phi'(t)dt\,.
$$
Further, the distribution function $\m_u$ is  a right-continuous and decreasing function and, moreover, it is  continuous if and only if $u^*$ is  strictly decreasing and 
\begin{equation}\label{mou1}
\m_u(u^*(s))=s\qquad  \mbox{   for a.e. } s\in (0,\Ln(\O)]\,.
\end{equation} 
Since $u^*$ is the distribution function of $\mu_u$,  $u^*$ is continuous if and only if $\mu_u$ is strictly decreasing, and in such case we have
\begin{equation}\label{mou2}
u^*(\m_u(t))=t\qquad  \mbox{   for a.e. } t\in (0,+\infty )\, 
\end{equation}
and
\begin{eqnarray*}
 & & \{(s,t)\in (0,\Ln(\O)) \times (0,+\infty): \>\>\> s \ge \mu_u(t)\}\\
&= & \{(s,t)\in (0,\Ln(\O)) \times (0,+\infty): \>\>\> t\ge u^*(s)\}\,.
\end{eqnarray*}
By Fubini's Theorem,  it  follows that
\begin{eqnarray*}
 & & \int_0^{+\infty} \left (\int_{\mu_u(t)}^{\Ln(\O)} (-W'(s))ds \right ) 
\phi'(t)\, dt \\
&=& \int_0^{\Ln(\O)} \left (\int_{u^*(s)}^{+\infty} \phi'(t)dt\right )(-W'(s))\, ds\\
&=&-\int_0^{+\infty}  \phi(u^*(s))(-W'(s))\, ds\,.
\end{eqnarray*}
Assumption \eqref{key} ensures that
\begin{eqnarray*}
 & & \int_0^{+\infty}  \phi(u^*(s))(-W'(s))\, ds\\ 
&=&\int_{[0,\Ln(\O)]\cap\{s: (-u^*)'(s)\neq 0\} }\phi(u^*(s))(-W'(s))\, ds\,.
\end{eqnarray*}
Therefore we obtain, by the coarea formula and  \eqref{dmut}, 
\begin{eqnarray*}
 & & \int_0^{+\infty}  \phi(u^*(s))(-W'(s))\, ds\\
&=& \int_0^{+\infty} \phi(t)\left (\int_{\{s: t=u^*(s)\}}(-W'(s))\frac 1 {(-u^*)'(s)}d\mathcal H^0\right )\, dt\\
&=&\int_0^{+\infty} (-W'(\mu_u(t))(-\mu_u'(t))\phi(t)\, dt\,.
\end{eqnarray*}
We deduce that $W\circ \mu_u$ has a distributional derivative which is given by
$$
(W\circ \mu_u)'(t)= (-W'(\mu_u(t))(-\mu_u'(t))\,,
$$
and \eqref{key} ensures that this derivative is bounded. This implies that $W\circ \mu_u$ is a Lipschitz continuous function. Observe that when $x\in \partial \O$ then $u(x)=0$ and $W(\mu_u(0))=W(\Ln(\O))=0$. By a classical result on composite functions in Sobolev spaces we conclude  that $(W\circ \mu_u)\circ u\in W^{1,p}_0(\O)$ and that its gradient can be evaluated through the chain rule.  This proves \eqref{Dw}.
\end{proof}

The last regularity result also allows to establish a nonlinear version of Theorem \ref{mainSch}. 

We will make use of the following nonlinear version of the classical P\'olya-Szeg\"o inequality (see e.g. \cite{BBMP},\cite{ET}, \cite{CiFe2}, \cite{Ch}, \cite{Br1}),

%\marginpar{controllare questa bibliografia in esposito trombetti}
\begin{equation}\label{PSnonlin}
\int_\O \mathcal A(u(x))|\nabla u(x)|^p\,dx\ge \int_{\O^\star} \mathcal A (u^\star(x))|\nabla u^\star(x)|^p\,dx\,,
\end{equation}
which holds for every nonnegative function $u\in W^{1,p}_0(\O)$, $1\le p < \infty$, and  for every bounded and Borel measurable function 
$\mathcal A:[0,+\infty ) \rightarrow [0,+\infty)$. 

\begin{theorem}\label{PSnonlin1}
Let $\O$ be a  bounded domain of $\RR^n$.  If $u\in W^{1,p}_0(\O)$, $1\le p < \infty$, and if $W$ is a function as in Lemma \ref{regularity}, 
%$W\in W^{1,1}(0,\Ln(\O))$ be a nonincreasing  function satisfying $W(\Ln(\O))=0$ and  \eqref{key}. 
then there exists only one function  $w\in W^{1,p}_0(\O)$  satisfying $w^{\star } = W$ and \eqref{HLS=}. Furthermore, there holds
\begin{equation}
\label{PSnonlin2}
\int_{\Omega} |\nabla u|^{p-2} \nabla u\cdot \nabla w\, dx\ge \int_{\Omega^\star}|\nabla u^\star|^{p-2}  \nabla u^\star \cdot \nabla w^\star\, dx\,.
\end{equation}
\end{theorem}
\begin{proof}
By Lemma \ref{regularity}, we deduce that there exists a unique extremal $w\in W_0 ^{1,1} (\Omega )$  of \eqref{HLS}, which can be represented by \eqref{rapw}.  Moreover, the gradient of $w$ is given by \eqref{Dw}. Since $W\circ \mu _u $ is Lipschitz continuous, this implies that $w\in W_0 ^{1,p} (\Omega )$. Hence we have that
 \begin{eqnarray}
\label{nonlin1}
 & & \int_{\Omega} |\nabla u|^{p-2} \nabla u\cdot \nabla w\, dx
\\
\nonumber 
&= & 
\int_{\Omega} W'(\mu_u(u(x)))\mu_u^\prime (u(x))|\nabla u|^{p} \, dx
=\int_{\Omega} \mathcal A (u(x))|\nabla u|^{p}dx\,,
\end{eqnarray}
where $\mathcal A(t)=W'(\mu_u(t))\mu_u^\prime(t)$ is a nonnegative bounded function. Applying \eqref{PSnonlin}, and since $\mu_u(u^\star (x))=\omega_n |x|^n$ and $w^\star (x)=W(\omega_n |x|^n)$, the assertion follows from  \eqref{nonlin1}.
% we immediately deduce the thesis.
\end{proof}
%\section{Comparison result for Schwarz symmetrization}\label{sec:comparison}

Using our method, 
we can recover a classical comparison result for Schwarz symmetrization which is due to Talenti (see \cite{Talentilin}).
\begin{theorem} Let $u\in H^{1}_0(\O)$ be a weak solution to problem \eqref{prSt}, and let $v\in H^{1}_0(\O^\star)$ the weak solution to the symmetrizated problem
\begin{equation*}\label{PstarSt}
\begin{cases}
 -\Delta v=f^\star &\mbox{ in } \O^\star\,,\\
 v=0 & \mbox{ on }\partial\O^\star\,,
\end{cases}
\end{equation*}
then
\begin{equation}\label{grad}
|\nabla u^\star(x)| \le |\nabla v(x)| \qquad \mbox{for  a.e. } x\in\O^\star \,.
\end{equation}
\end{theorem}
\begin{proof} 
We adapt the previous proof  to the case of Schwarz symmetrization. Since the vectors $\nabla u^\star$, $\nabla v$ and $\nabla W$ are parallel, the inequality \eqref{st3} is equivalent to
$$
\int_{\O^\star}\left [ |\nabla u^\star(x)|-|\nabla v(x)|\right ]|\nabla W(x)|dx\le 0\,.
$$
By the arbitrariness of $W$, \eqref{grad} follows.
\end{proof}
Finally, a similar result can be also obtained for nonlinear differential operators. More precisely, we consider the following homogeneous Dirichlet problem for the $p$-Laplacian, 
\begin{equation} \label{pr}
\left\{
\begin{array}{lll}
-\mbox{div}\left(  | \nabla u|^{p-2} \nabla u \right) =f &  & \text{in}\ \Omega \\
u=0 &  & \text{on}\ \partial \Omega\,,%
\end{array}%
\right. 
\end{equation}%
where $\Omega $ is an bounded domain of $\mathbb{R}^{N}$, $N\ge 2$, $1<p<\infty$ and $f$ is a measurable function belonging to 
$L^{(p ^*)'}(\Omega)$, where $(p^*)' := \frac{np}{np-n+p} $.

The Polya-Sz\"ego type inequality  proved in the previous section allows us to give a different proof of a comparison result which is due to Talenti (see \cite{Talentinonlin}).

\begin{theorem}
\label{compnonlin}
Let $u\in W_0 ^{1,p}(\Omega)$ be the weak solution to problem \eqref{pr} and let $v\in W_0 ^{1,p}(\Omega^\star)$ be the weak solution to the symmetrized problem
\begin{equation}\label{Pstar}
\begin{cases}
-\mbox{\rm div}\left( | \nabla v|^{p-2}\nabla v  \right) =f^\star &\mbox{ in } \O^\star\,,\\
 v=0 & \mbox{ on }\partial\O^\star\,.
\end{cases}
\end{equation}
Then
$$
u^\star (x)\le  v(x) \quad \mbox{ for a.e. } x\in\O\,.
$$
\end{theorem}
\begin{proof}
%First observe that the problems (\ref{pr}) and (\ref{Pstar}) have indeed unique solutions.
%Its is well-known that
%First we suppose that $f\in C^\infty (\O)$. 
%Then there exists a unique solution $u\in C^\infty (\O)\cap C^{1,\alpha}(\O)$, $0<\alpha<1$, %to the problem \eqref{pr}. 
Let $h\in L^\infty(0,+\infty)$ be a nonnegative function and consider the function
$$ 
\Phi (t)=\int_0^t h(\tau)d\tau\,, \qquad t\ge 0\,.
$$
Since $\Phi$ is  Lipschitz continuous and $\Phi(0)=0$, we can choose $w=\Phi(u)$ as a test function in \eqref{pr} to get
\begin{equation}\label{comp1}
\int_\O |\nabla u|^{p-2}\nabla u\cdot\nabla w\,dx=\int_\O fwdx\,.
\end{equation}
Further, since $\Phi$ is non decreasing, we have $[\Phi(u)]^*=\Phi(u^*)$ and, arguing as above, we can choose $w^\star=\Phi(u^*)$ as test function in \eqref{Pstar}. By \eqref{comp1} and the Hardy-Littlewood inequality we obtain
%\begin{align}
\begin{equation}\label{comp2}
\int_\O |\nabla u|^{p-2}\nabla u\cdot\nabla w\,dx
%&
\le \int_{\O^\star}f^\star w^\star dx
%\notag \\
%&
=\int_{\O^\star}|\nabla v|^{p-2}\nabla v\cdot\nabla w^\star dx\,.
\end{equation}
%\end{align}
On the other hand, the function $W:= w^\star$ satisfies the assumptions of Theorem \ref{PSnonlin1} and $w$ is  the function which realizes equality \eqref{HLS=}. Hence, applying  \eqref{comp2}, we have
\begin{equation}\label{comp3}
\int_{\O^\star} |\nabla u^\star|^{p-2}\nabla u^\star\cdot\nabla w^\star\,dx
\le \int_{\O^\star}|\nabla v|^{p-2}\nabla v\cdot\nabla w^\star dx\,.
\end{equation}
In view of the coarea formula and 
since  
$$|\nabla w^\star(x)|=\Phi'(u^\star)|\nabla u^\star (x)|=h(u^\star)|\nabla u^\star (x)|\, ,$$ 
 equation \eqref{comp3} can be equivalently written as
\begin{eqnarray}
\label{comp4}
 & & \qquad \int_{\O^\star} h(u^\star(x))|\nabla u^\star (x)|\left [ |\nabla u^\star(x)|^{p-1}-|\nabla v(x)|^{p-1}\right ]dx\\
\nonumber
 &= & \int_0^{+\infty} h(t)\int_{\{x: \omega_n|x|^n=\mu_u(t)\}}\left [ |\nabla u^\star(x)|^{p-1}-|\nabla v(x)|^{p-1}\right ] d\HH^{n-1}\,dt
\le 0\,.
\end{eqnarray}
Furthermore, since both $u^\star$ and $v$ are radial functions, $|\nabla u^\star(x)|$ and $|\nabla v(x)|$ are constant on $\{x\in \O^\star\,: \omega_n|x|^n=\mu_u(t)\}$ and these constants are given by
$$|\nabla u^\star |_{|\{x: \omega_n|x|^n=\mu_u(t)\}}=n\omega_n^{1/n}\mu_u(t)^{1-1/n}\left (-\frac {du^*}{ds}(\mu_u(t))\right )\,, $$
$$|\nabla v |_{|\{x: \omega_n|x|^n=\mu_u(t)\}}=n\omega_n^{1/n}\mu_u(t)^{1-1/n}\left (-\frac {dv^*}{ds}(\mu_u(t))\right )\,.$$
Together with (\ref{comp4}) this implies 
$$\int_0^{+\infty} h(t)\, \left [ \left (-\frac {du^*}{ds}(\mu_u(t))\right )^{p-1}-\left (-\frac {dv^*}{ds}(\mu_u(t))\right )^{p-1}\right ] \left (n\omega_n^{1/n}\mu_u(t)\right )^pdt\le 0\,,$$
for any nonnegative function $h\in L^\infty (0,+\infty)$. By the arbitrariness of $h$, we deduce
$$\left (-\frac {du^*}{ds}(\mu_u(t))\right )^{p-1}\le \left (-\frac {dv^*}{ds}(\mu_u(t))\right )^{p-1}\,, \quad \mbox{for a.e. } t\in(0,\infty)\,,$$
from which we easily deduce the thesis. 
%IN THE CASE WHERE $f$ is smooth enough.
\end{proof}

\begin{thebibliography}{99}        
                                                                           
\bibitem {ADLT} 
 \newblock {\sc A. Alvino, G. Trombetti, J.I. Diaz, P.-L. Lions}, 
 \newblock {\emph {Elliptic equations and {S}teiner symmetrization}}, 
 \newblock   Comm. Pure Appl. Math. \textbf{49} (1996), no. 3, 217--236.
 
  \bibitem {ALT1} 
 \newblock {\sc A. Alvino, P.-L. Lions, G. Trombetti}, 
 \newblock {\emph {On optimization problems with prescribed rearrangements}}, 
 \newblock   Nonlinear Anal. \textbf{13} (1989), no. 2, 185--220.
                                                 
\bibitem {ALT2} 
 \newblock {\sc A. Alvino, P.-L. Lions, G. Trombetti}, 
 \newblock {\emph {Comparison results for elliptic and parabolic equations via symmetrization: a new approach}}, 
 \newblock   Differential Integral Equations  \textbf{4} (1991), no. 1, 25--50.


 \bibitem {BaeTa}  
  \newblock {\sc A. Baernstein II, B.A. Taylor}, 
 \newblock {\emph {Spherical rearrangements, subharmonic functions, and {$\sp*$}-functions in {$n$}-space}}, 
 \newblock   Duke Math. J.  \textbf{43} (1976), no. 2, 245--268.
 
 \bibitem {Bae}  
  \newblock {\sc A. Baernstein II}, 
 \newblock {\emph {A unified approach to symmetrization.}}, 
 \newblock   Partial differential equations of elliptic type (Cortona, 1992), 47--91, Sympos. Math., XXXV, Cambridge Univ. Press, Cambridge, 1994.

\bibitem{Ba} \textsc{C. Bandle}, \textsl{Isoperimetric inequalities and
applications}. Monographs and Studies in Mathematics \textbf{7}, Pitman
(Advanced Publishing Program), Boston, Mass.-London, 1980.

\bibitem {BK} 
 \newblock {\sc C. Bandle, B. Kawohl}, 
 \newblock {\emph  {Applications de la sym\' etrization de Steiner aux probl\`emes de Poisson}}, 
 \newblock   Preprint (1992).


\bibitem {BLM} 
 \newblock {\sc H.J. Brascamp, E.H. Lieb, J.M. Luttinger}, 
 \newblock {\emph {A general rearrangement inequality for multiple integrals}}, 
 \newblock   J. Functional Analysis \textbf{17} (1974), 227--237.

 \bibitem {Br1} 
 \newblock {\sc F. Brock}, 
 \newblock  {\emph {Rearrangements and applications to symmetry problems in 
 PDE}}, 
 \newblock  {"Handbook of differential equations: stationary partial differential equations. {V}ol. {IV}}, 
{ \emph{Elsevier/North-Holland}}, Amsterdam, (2007), 1--60.
   
 
 \bibitem {BBMP} 
 \newblock {\sc F. Brock, M.F. Betta, A. Mercaldo, M.R. Posteraro}, 
 \newblock {\emph {A weighted isoperimetric inequality and applications to symmetrization.}}, 
 \newblock  J. Inequal. Appl.  \textbf{4} (1999), no. 3,  215�240. 

 
 \bibitem {BMP} 
 \newblock {\sc F. Brock, A. Mercaldo, M.R. Posteraro}, 
 \newblock {\emph {On isoperimetric inequalities with respect to infinite measures.}}, 
 \newblock  Rev. Mat. Iberoam.  \textbf{29} (2013), no. 2,  665--690. 
 
  
   
 \bibitem {BZ} 
 \newblock {\sc J.E. Brothers, W. Ziemer}, 
 \newblock {\emph {Minimal rearrangements of Sobolev functions}}, 
 \newblock   Acta Univ. Carolin. Math. Phys. \textbf{28} (1987), no. 2, 13--24.

 \bibitem {Bur1} 
 \newblock {\sc A. Burchard}, 
 \newblock {\emph {Steiner symmetrization is continuous in $W^{1,p}$}}, 
 \newblock   Geom. Funct. Anal.  \textbf{7} (1997), no. 5, 823--860.
 
 \bibitem {Bur2} 
 \newblock {\sc A. Burchard}, 
 \newblock {\emph {Cases of equality in the Riesz rearrangement 
 inequality}}, 
 \newblock   Ann. of Math. (2)  \textbf{143} (1996), no. 3, 499--527.
 
  \bibitem {ChRic}
 \newblock {\sc F. Chiacchio}, 
 \newblock {\emph {Steiner symmetrization for an elliptic problem with lower-order terms.}}, 
\newblock    Ricerche Mat.\textbf{53} (2005), no. 1, 87�106. 

 
  \bibitem {Ch}
 \newblock {\sc F. Chiacchio}, 
 \newblock {\emph {Estimates for the first eigenfunction of linear eigenvalue problems via Steiner symmetrization}}, 
\newblock   Publ. Mat. \textbf{53} (2009), no. 1, 47--71. 

 
   
       
 \bibitem {C1}
 \newblock {\sc A. Cianchi}, 
 \newblock {\emph {On the $L^q$ norm of functions having equidistributed gradients}}, 
 \newblock   Nonlinear Anal.  \textbf{26} (1996), no. 12, 2007--2021.
 
 \bibitem {CiFe1}
 \newblock {\sc A. Cianchi, A. Ferone}, 
 \newblock {\emph {A strengthened version of the Hardy-Littlewood inequality}}, 
 \newblock   J. Lond. Math. Soc. (2)  \textbf{77} (2008), no. 3, 581--592.
 
\bibitem {CiFe2}
 \newblock {\sc A. Cianchi, A. Ferone}, 
 \newblock {\emph {On symmetric functionals of the gradient having symmetric equidistributed minimizers}}, 
 \newblock   SIAM J. Math. Anal.  \textbf{38} (2006), no. 1, 279--308.
 
 \bibitem {CF1} 
 \newblock {\sc A. Cianchi, N. Fusco}, 
 \newblock {\emph {Functions of bounded variation and rearrangements}}, 
 \newblock   Arch. Ration. Mech. Anal.
  \textbf{165} (2002), no. 1, 1--40.

 \bibitem {CM} 
 \newblock {\sc A. Cianchi, V. Maz'ya}, 
 \newblock {\emph {Gradient regularity via rearrangements for {$p$}-Laplacian
              type elliptic boundary value problems}}, 
 \newblock   J. Eur. Math. Soc. (JEMS)  \textbf{16} (2014), no. 3, 571--595.

 \bibitem {E} 
 \newblock {\sc A. Ehrhard}, 
 \newblock {\emph { In\' egalit\' es isop\' erim\' etriques et int\' egrales de Dirichlet gaussiennes.}}, 
 \newblock  Ann. Sci. \' Ecole Norm. Sup. (4) \textbf{17} (1984), no. 2, 317--332.
 
\bibitem {ET} 
 \newblock {\sc L. Esposito, C. Trombetti}, 
 \newblock {\emph {Convex symmetrization and {P}{\'o}lya-{S}zeg{\"o} inequality}}, 
 \newblock   Nonlinear Anal.  \textbf{56} (2004), no. 1, 43--62.
   
 \bibitem {FM} 
 \newblock {\sc V. Ferone, A. Mercaldo}, 
 \newblock {\emph {A second order derivation formula for functions defined by integrals.}}, 
 \newblock   C. R. Acad. Sci. Paris Ser. I Math. \textbf{326} (1998), no. 5, 549--554.


\bibitem {FM2} 
 \newblock {\sc V. Ferone, A. Mercaldo}, 
 \newblock {\emph {Neumann problems and Steiner symmetrization.}}, 
 \newblock   Comm. Partial Differential Equations \textbf{30} (2005), no. 10-12, 1537�1553.

 
 \bibitem {FV} 
 \newblock {\sc A. Ferone, R. Volpicelli}, 
 \newblock {\emph {Minimal rearrangements of Sobolev functions: a new proof.}}, 
 \newblock    Ann. Inst. H. Poincar\'e Anal. Non Lin\' eaire  \textbf{20} (2003),  no. 2, 333--339. 
 
 \bibitem {He} 
\newblock {\sc A. Henrot}, 
 \newblock {\emph { Extremum problems for eigenvalues of
elliptic operators}}, 
 \newblock   Frontiers in Mathematics. Birkh\"{a}user Verlag, Basel,
2006.

 \bibitem {Ka} 
\newblock {\sc B. Kawohl}, 
 \newblock {\emph { Rearrangements and Convexity of
Level Sets in PDE}}, 
 \newblock   Lecture Notes in Mathematics \textbf{1150}, New York:
Springer Verlag, 1985.

 \bibitem {Ke} 
\newblock {\sc S. Kesavan}, 
 \newblock {\emph { Symmetrization \& applications.}}, 
 \newblock   Series in Analysis, 3. World Scientific Publishing Co. Pte. Ltd., Hackensack, NJ, 2006.
 

 \bibitem {Talentilin} 
 \newblock {\sc G. Talenti}, 
 \newblock {\emph {Elliptic equations and rearrangements. }}
 \newblock   Ann. Scuola Norm. Sup. Pisa Cl. Sci. (4)  \textbf{3} (1976), 697�718.


\bibitem {Talentinonlin} 
 \newblock {\sc G. Talenti}, 
 \newblock {\emph {Nonlinear elliptic equations, rearrangements of functions and {O}rlicz spaces}}, 
 \newblock   Ann. Mat. Pura Appl. (4)  \textbf{120} (1979), 160--184.

\bibitem {Talentiweig} 
 \newblock {\sc G. Talenti}, 
 \newblock {\emph {A weighted version of a rearrangement inequality}}, 
 \newblock   Ann. Univ. Ferrara Sez. VII (N.S.) (4),  \textbf{43} (1997), 121--133.
 
  \bibitem {Tart} 
 \newblock {\sc G. Talenti}, 
 \newblock {\emph {The art of rearranging}}, 
 \newblock   Milan J. Math.   \textbf{84} (2016), 105--157.


\end{thebibliography}
\end{document}